\documentclass{amsart}

\usepackage{mathpazo,amssymb}
\usepackage{url}

\usepackage[T1]{fontenc}
\usepackage[utf8]{inputenc}

\author{{\L}ukasz Garncarek}

\title{Property of rapid decay for extensions of compactly generated groups}


\address{University of Wroc{\l}aw, Institute of Mathematics,
  pl.~Grunwaldzki 2/4, 50-384 Wroc\-{\l}aw, Poland}

\email{Lukasz.Garncarek@math.uni.wroc.pl}

\thanks{During part of the work on this paper the author was supported
  by a scholarship of the Foundation for Polish Science. Part of this
  work was conducted during author's internship at the Warsaw Center
  of Mathematics and Computer Science. Research partially supported by
  the contract with University of Wrocław 2139/M/IM/12 and by the
  Polish National Science Center grant 2012/06/A/ST1/00259}

\newtheorem{theorem}{Theorem}[section]
\newtheorem{lemma}[theorem]{Lemma}
\newtheorem{proposition}[theorem]{Proposition}
\newtheorem{corollary}[theorem]{Corollary}
\theoremstyle{definition}

\newtheorem*{problem}{Problem}
\theoremstyle{remark}
\newtheorem{remark}[theorem]{Remark}

\newcommand{\R}{\mathbb{R}}
\newcommand{\Z}{\mathbb{Z}}

\DeclareMathOperator{\Aut}{Aut}

\newcommand{\norm}[1]{\left\lVert#1\right\rVert}
\newcommand{\abs}[1]{\left\lvert#1\right\rvert}

\newcommand{\then}{\Rightarrow}

\begin{document}

\begin{abstract}
  In the article we settle down the problem of permanence of property
  RD under group extensions. We show that if $1\to N\to G\to Q\to 1$
  is a short exact sequence of compactly generated groups such that
  $Q$ has property RD, and $N$ has property RD with respect to the
  restriction of a word-length on $G$, then $G$ has property RD. 

  We also generalize the result of Ji and Schweitzer stating that
  locally compact groups with property RD are unimodular. Namely, we
  show that any automorphism of a locally compact group with property
  RD which distorts distances subexponentially, preserves the Haar measure.
\end{abstract}

\maketitle

\section{Introduction}
\label{sec:introduction}

The property of rapid decay (property RD for short) is a certain
estimate for the norm of convolution operators in terms of a
Sobolev-type norm corresponding to a length function on a group. Its
study was begun by Haagerup, who established it for free groups in
\cite{Haagerup1978}. The basic theory was later developed by
Jolissaint in \cite{Jolissaint1989,Jolissaint1990}. 

Property RD takes its name from an equivalent formulation, which
states that the space of rapidly decreasing (with respect to some
length) functions on a group naturally embeds in its reduced
$C^*$-algebra. It turns out that this embedding induces isomorphisms
in $K$-theory. This result found an application in the
Connes--Moscovici proof of the Novikov conjecture for hyperbolic
groups \cite{Connes1990}.


In the paper \cite{Jolissaint1990} a partial result on permanence of
property RD under group extensions is proved. Proposition 2.1.9
therein treats a special case of an extension $1\to N\to G\to Q \to 1$
of finitely generated groups, satisfying technical conditions of
polynomial amplitude and polynomial growth of certain associated
functions. In Section~\ref{sec:final-remarks} we show that these
assumptions are equivalent to $N$ being polynomially distorted in
$G$. Other related results are \cite[Proposition 1.14]{ChatterjiPhD},
dealing with split extensions of finitely generated groups, and
\cite[Proposition 5.5]{Chatterji2007}, dealing with polynomially
distorted central extensions of compactly generated groups.

The main aim of this article is to provide a permanence result for
general extensions of compactly generated groups. We show that for a
compactly generated group $G$ to have property RD, it is sufficient
that it contains a normal subgroup $N$ satisfying property RD with
respect to the restriction of a word-length of $G$, such that the
quotient $G/N$ has property RD. The proof is an adaptation of the
proof of \cite[Proposition 2.1.9]{Jolissaint1990}, based on a more
careful choice of an auxiliary cross-section of the short exact
sequence.

In the course of the proof, we first need to show that our extension
is unimodular. We achieve this by proving a generalization of the
result of Ji and Schweitzer \cite[Theorem 2.2]{Ronghui1996} stating
that groups with property RD are unimodular. This is equivalent to
saying that inner automorphisms are measure-preserving, and we extend
this statement to all automorphisms which distort the length
subexponentially in a certain sense.

The text is organized as follows. Section~\ref{sec:property-rd}
introduces the basic notions associated to property RD. In
Section~\ref{sec:unimodularity} we prove an inequality about length
distortion of automorphisms of groups with property RD, and use it to
show that an extension satisfying the assumptions of our main theorem
is unimodular. Section~\ref{sec:permanence} contains the proof of the
main result, while Section~\ref{sec:final-remarks} is devoted to the
discussion of some examples and the relation of our work to that of
Jolissaint.

\section{Property RD}
\label{sec:property-rd}

Let $G$ be a locally compact group. We will always endow groups with
their right-invariant Haar measures. A \emph{length} on $G$ is a Borel
function $\ell\colon G\to [0,\infty)$ such that
\begin{enumerate}
\item $\ell(1)=0$,
\item $\ell(x^{-1})=\ell(x)$,
\item $\ell(xy)\leq\ell(x)+\ell(y)$.  
\end{enumerate}
If $G$ is generated by a subset $S$, then the word-length
\begin{equation}
\ell(x)=\min\{n : x\in (S^{\pm1})^n\}
\end{equation}
is an example of a length. When speaking about word-lengths we will
always assume that the set $S$ is relatively compact, and in
particular, that the group is compactly generated.

Suppose $\ell_1$ and $\ell_2$ are two lengths on $G$. We say that $\ell_1$
\emph{dominates} $\ell_2$ if there exist constants $r,C>0$ such that
\begin{equation}
  \label{eq:def-dominates}
  \ell_2(x) \leq C(1+\ell_1(x))^r
\end{equation}
for all $x\in G$. If $\ell_2$ also dominates $\ell_1$, they are said
to be \emph{equivalent}. By \cite[Theorem 1.2.11]{Schweitzer1993}, any
length function is bounded on compact sets. Therefore, if $G$ is
compactly generated, all its word-lengths are equivalent and dominate
all other lengths.

The algebra $C_c(G)$ of compactly supported continuous functions on
$G$ acts faithfully on the Hilbert space $L^2(G)$ by left convolution
operators $T_f$, where
\begin{equation}
  \label{eq:1}
  T_fg(x) = \int_G f(y^{-1})g(yx)\,dy.
\end{equation}
This induces the operator norm $
$$\norm{\cdot}_{op}$ on $C_c(G)$. The group $G$ is said to satisfy
\emph{property RD with respect to a length function $\ell$} if there
exist constants $s,C>0$ such that for every $f\in C_c(G)$
 \begin{equation}
  \label{eq:def-RD}
  \norm{f}_{op}\leq C(1+\ell(f))^s \norm{f}_2,  
\end{equation}
where $\ell(f)=\sup\{\ell(x) : f(x)\ne 0\}$, and $\norm{\cdot}_2$
stands for the  $L^2$-norm. We will also later employ the notation
$\ell(U) = \sup\{\ell(x) : x\in U\}$ for $U\subseteq G$. 

If $\ell_1$ and $\ell_2$ are two lengths on $G$ such that $\ell_1$
dominates $\ell_2$ and $G$ has property RD with respect to $\ell_2$,
then clearly it has property RD with respect to $\ell_1$.  Hence, if a
compactly generated group has property RD with respect to one length,
then it has property RD with respect to any of its word-lengths. In
this case, it is said that $G$ has \emph{property RD}.
Finally, since by \cite{Ronghui1996} property RD implies
unimodularity, using a left Haar measure leads to exactly the same notion.

\section{Unimodularity of the extension}
\label{sec:unimodularity}

In this section we exhibit an inequality satisfied by automorphisms of
groups with property RD, which, when applied to inner automorphisms,
generalizes the result of Ji and Schweitzer, stating that groups with
property RD are unimodular \cite{Ronghui1996}. We use it to infer
that in an extension $1\to N\to G\to Q\to 1$, where $Q$ has property
RD, and $N$ has property RD with respect to the restriction of the
word-length of $G$, the group $G$ is unimodular. This allows to
apply a lemma of Jolissaint to show that $G$ has property RD.

For a locally compact group $G$ we will denote by $\Aut(G)$ the group
of all topological automorphisms of $G$. If $\rho$ is a right Haar
measure on $G$, then the modular homomorphism
$\Delta_G\colon\Aut(G)\to \R_+$ is defined by $\alpha_*\rho =
\Delta_G(\alpha)\rho$.

\begin{theorem}
  \label{thm:rd-obstruction}
  Suppose that a locally compact group $G$ satisfies property RD with respect to a length $\ell$
  with exponent $s$. Let $\alpha\in\Aut(G)$. Then for any relatively
  compact open set $U\subseteq G$ there exists $D>0$ such that
  \begin{equation}
    \label{eq:distortion-comparison}
    \Delta_G(\alpha)^{n} \leq D(1+
    \ell(\alpha^{-n}(U)))^{2s} 
  \end{equation}
  for all $n\in\Z$.
\end{theorem}

\begin{proof}
  Observe that for $f,g\in C_c(G)$
  \begin{equation}
    \label{eq:conv-identity}
    (f\circ\alpha)* (g\circ\alpha) = \Delta_G(\alpha)
    (f*g)\circ\alpha,
  \end{equation}
  and 
  \begin{equation}
    \label{eq:aut-norm}
    \norm{f\circ\alpha}_2 = \Delta_G(\alpha)^{1/2}\norm{f}_2.
  \end{equation}
  It follows that 
  \begin{equation}
    \label{eq:conv-ineq}
    \begin{split}
      \Delta_G(\alpha)^{n/2} \lVert f*g\rVert_2= &
      \lVert{(f*g)\circ\alpha^n}\rVert_2 = \Delta_G(\alpha)^{-n}
      \lVert{(f\circ\alpha^n)*(g\circ\alpha^n)\rVert}_2 \leq\\
      \leq & C(1+\ell(f\circ\alpha^n))^s \Delta_G(\alpha)^{-n} \lVert f\circ
      \alpha^n \rVert_2 \lVert g\circ\alpha^n\rVert_2 =\\
      =& C(1+\ell(f\circ\alpha^n))^s \norm{f}_2\norm{g}_2.
    \end{split}
  \end{equation}
  For arbitrarily chosen positive $f,g\in C_c(G)$ with $f$ vanishing
  outside $U$, the convolution $f*g$ is nonzero, and
  \begin{equation}
    \Delta_G(\alpha)^n \leq (1+\ell(f\circ\alpha^n))^{2s} 
    \left( \frac{C\norm{f}_2\norm{g}_2}{\norm{f*g}_2}\right)^2,
  \end{equation}
  which, after noticing that $\ell(f\circ\alpha^n) \leq
  \ell(\alpha^{-n}(U))$, yields the desired inequality.
\end{proof}

If $\alpha(x)=axa^{-1}$ is an inner automorphism of $G$, then we have
$\ell(\alpha^n(x)) \leq \ell(x)+2\abs{n}\ell(a)$. By
Theorem~\ref{thm:rd-obstruction}, if $G$ has property RD, then the
function $(\rho(\alpha(U))/\rho(U))^n$ is bounded by a polynomial in
$\lvert n\rvert$, and hence $\alpha$ is measure-preserving. Thus this
theorem indeed generalizes the result on unimodularity of groups with
property RD.

Now, consider a short exact sequence 
 $1\to N\to G\overset{\pi}{\to} Q\to 1$ of compactly generated groups. By \cite[Lemma
1.1]{Mackey1952}, there exists a Borel cross-section $\sigma\colon
Q\to G$. There are two Borel functions associated to the choice of
$\sigma$, namely $\beta\colon Q\times Q\to N$ and $\theta\colon
Q\to\Aut(N)$, defined by
\begin{equation}
  \label{eq:def-f}
  \beta(p,q)=\sigma(p)\sigma(q)\sigma(pq)^{-1}  
\end{equation}
and
\begin{equation}
  \label{eq:def-rho}
  \theta(q)(n)=\sigma(q)n\sigma(q)^{-1}.  
\end{equation}
Using $\sigma$ we obtain a Borel identification $N\times Q\to G$ given by
$(n,q)\mapsto n\sigma(q)$, which we will use freely without any further
mention. The multiplication on $G$ can be now expressed by
\begin{equation}
  \label{eq:def-mult-ext}
  (m,p)(n,q)=(m\theta(p)(n)\beta(p,q),pq).  
\end{equation}
Moreover, the right Haar measure of $G$ is the product of right Haar
measures of $N$ and $Q$.

\begin{lemma}
  \label{thm:unimodularity-conditions}
  Suppose that $1\to N\to G \to Q\to 1$ is a short exact sequence of
  compactly generated groups. If $N$ and $Q$ are unimodular, then $G$
  is unimodular if and only if the automorphisms $\theta(q)$ defined
  by~(\ref{eq:def-rho}) are measure-preserving.
\end{lemma}

\begin{proof}
  Let $f\in C_c(G)$. For $(m,p)\in G$ we have
  \begin{equation}
    \begin{split}
      \int_N\int_Q & f((m,p)(n,q))\,dqdn = \int_N\int_Q f
      (m\theta(p)(n)\beta(p,q),pq)\,dqdn = \\
      = & \int_N\int_Q f(\theta(p)(n),q)\,dqdn =
      \Delta_N(\theta(p))\int_N\int_Q f(n,q)\,dqdn,
    \end{split}
  \end{equation}
  so the right Haar measure on $G$ is left-invariant if and only if
  $\Delta_n(\theta(p))=1$ for all $p\in Q$.
\end{proof}

\begin{corollary}
  \label{thm:unimodularity-extension}
  If in the short exact sequence $1\to N\to G \to Q\to 1$ of compactly
  generated groups $Q$ is unimodular, and $N$ has property RD with respect to the
  restriction of a word-length of $G$, then $G$ is unimodular.
\end{corollary}

\begin{proof}
  By Lemma~\ref{thm:unimodularity-conditions}, it is enough to observe
  that the automorphisms $\theta(q)$ associated with the section
  $\sigma$ are measure-preserving. For $n\in N$ we have
  \begin{equation}
    \ell_G(\theta(q)^{-k}(n)) \leq \ell_G(n) + 2\abs{k} \ell_G(\sigma(q)),
  \end{equation}
  and thus by Theorem~\ref{thm:rd-obstruction}, the sequence
  $\Delta_N(\theta(q))^k$, with $k\in\Z$, is bounded by a polynomial
  in $\abs{k}$. This is only possible if $\Delta_N(\theta(q))=1$.
\end{proof}

\section{Permanence of property RD under extensions}
\label{sec:permanence}

Again, let $1\to N\to G\overset{\pi}{\to} Q\to 1$ be a short exact
sequence of compactly generated groups. Endow $G$ with a word-length
$\ell_G$ corresponding to a relatively compact generating set
$S$. Denote by $\ell_Q$ the word-length on $Q$ corresponding to
$\pi(S)$. The main result of this article is the following theorem,
the proof of which we postpone until the end of this section.

\begin{theorem}
  \label{thm:main-thm}
  If in the short exact sequence $1\to N\to G \to Q\to 1$ of
  compactly generated groups $Q$ has property RD with respect to its
  word-length, and $N$ has property RD with respect to the restriction
  of the word-length of $G$, then $G$ has property RD.
\end{theorem}

We have already remarked that there exists a Borel section
of the quotient map $\pi$. The proof in \cite{Mackey1952} actually
yields more. The next simple observation will be crucial---choosing
the right cross-section will allow to drop the unnecessary assumptions
of \cite[Proposition 2.1.9]{Jolissaint1990}.

\begin{remark}
  \label{thm:good-section}
  There exists a Borel cross-section $\sigma\colon Q\to G$ of $\pi$ such that
  for every $q\in Q$ we have $\ell_G(\sigma(q)) =\ell_Q(q)$.
\end{remark}

Indeed, since we have $\ell_Q(\pi(g)) \leq \ell_G(g)$, and in
consequence $\ell_Q(q)\leq\ell_G(\sigma(q))$, we need to construct a
cross-section $\sigma$ such that $\sigma(\pi(S)^k)\subseteq S^k$. But
the proof of Mackey actually follows by first decomposing $G$ into an
increasing union of compact sets $K_n$ such that every compact subset
$K\subseteq G$ is contained in some $K_n$, and then constructing an
increasing family of sections of the restrictions $\pi|_{K_n}$, using
the Federer-Morse theorem. We may therefore put
$K_n=S^n$, and we just need to notice that every compact subset
$K\subseteq G$ is contained in $S^n$ for some $n$. This is clear, as
$S^k$ has positive measure for some $k$, and the convolution
$\chi_{S^k}*\chi_{S^k}$ is nonzero, continuous (as convolution of
square-integrable functions), and vanishing outside $S^{2k}$. Hence
$S^{2k}$ has nonempty interior and generates $G$, so the claim follows.

The next ingredient in the proof of Theorem~\ref{thm:main-thm} is a
generalization of \cite[Lemma 2.1.2]{Jolissaint1990}, which was
formulated in the setting of discrete groups. The proof, adapted
from~\cite{Jolissaint1990}, works for arbitrary lengths on $G$ and
$Q$, not only for word-lengths.

\begin{lemma}
  \label{thm:length-on-extension} Let $1\to N\to G\to
  Q\overset{\pi}{\to} 1$ be a short exact sequence of compactly
  generated groups, endowed with lengths $\ell_N$, $\ell_G$, and
  $\ell_Q$, where $\ell_N$ is arbitrary, $\ell_G$ is a word-length
  corresponding to a relatively compact generating set $S$, and
  $\ell_Q$ is the word-length corresponding to $\pi(S)$. Suppose that
  \begin{enumerate}
  \item $N$ has property RD with respect to $\ell_N$,
  \item $Q$ has property RD with respect to $\ell_Q$,
  \item $G$ is unimodular,
  \item there exist constants $D,r>0$ such that for any $(n,q)\in G$
    we have
    \begin{equation}
      \ell_N(n)+ \ell_Q(q) \leq D\ell_G(n,q)^r.
    \end{equation}
  \end{enumerate}
  Then $G$ has property RD with respect to $\ell_G$.
\end{lemma}

\begin{proof}
  Let $f,g\in C_c(G)$. We get, using unimodularity of $N$ and $G$, that
  \begin{equation}
    \begin{split}
      f*g&(n,q) = \int_Q\int_N f((m,p)^{-1})g((m,p)(n,q))\,dmdp = \\
      = & \int_Q\int_N f(\theta(p)^{-1}(m^{-1}\beta(p,p^{-1})^{-1}),
      p^{-1}) g(m\theta(p)(n)\beta(p,q),pq)\,dmdp =\\
      = & \Delta_N(\theta(p)^{-1})\int_Q\int_N f_p(m^{-1})
      g_{p,q}(mn)\,dmdp = \int_Q f_p*g_{p,q}(n)\,dp,
    \end{split}
  \end{equation}
  where
  \begin{equation}
    f_p=f(m,p^{-1}),
  \end{equation}
  and
  \begin{equation}
    g_{p,q}(m) = g(\beta(p,p^{-1})^{-1}\theta(p)(m)\beta(p,q),pq).
  \end{equation}
  Now, let $N$ and $Q$ satisfy property RD with constants $C$ and
  $s$. Using the triangle inequality, we may estimate the norm of
  $f*g$ by
  \begin{equation}
    \begin{split}
      \norm{f*g}_2^2 & \leq  \int_Q \left[ \int_Q \lVert f_p * g_{p,q}
        \rVert_2 \,dp \right]^2\, dq \leq \\ 
      & \leq  \int_Q\left[ C \int_Q (1 + \ell_N(f_p))^s\lVert f_p
        \rVert_2 \lVert g_{p,q} \rVert_2 \,dp\right]^2 \,dq \leq \\
      & \leq  C^2(1+D\ell_G(f)^r)^{2s} \int_Q \left[ \int_Q \lVert
        f_p\rVert_2 \lVert g_{p,q} \rVert_2\, dp\right]^2 \,dq
    \end{split}
  \end{equation}
  If we put
  \begin{equation}
    \phi(p) = \lVert{ f_{p^{-1}}}\rVert_2
  \end{equation}
  and
  \begin{equation}
    \psi(p) = \left[\int_N \lvert g(m,p)\rvert^2\,dm\right]^{1/2},
  \end{equation}
  we have, using unimodularity of $N$, that $\lVert g_{p,q}\rVert_2 =
  \psi(pq)$, and therefore
  \begin{equation}
    \begin{split}
      \norm{f*g}_2 & \leq C(1+D\ell_G(f)^r)^{s}\lVert \phi * \psi
      \rVert_2 \leq \\
      & \leq
      C^2(1+D\ell_G(f)^r)^{s}(1+\ell_Q(\phi))^s\norm{\phi}_2\norm{\psi}_2
      \leq \\
      & \leq C^2 (1+D\ell_G(f)^r)^{2s}\norm{f}_2\norm{g}_2 \leq
      C'(1+\ell_G(f))^{2rs} \norm{f}_2\norm{g}_2
    \end{split}
  \end{equation}
  for some constant $C'>0$.
\end{proof}

Now, it turns out that if we use the cross-section from
Remark~\ref{thm:good-section}, the inequality in
Lemma~\ref{thm:length-on-extension} is satisfied by the restriction of
the word-length of $G$.

\begin{lemma}
  \label{thm:final-length-inequality}
  If in the short exact sequence $1\to N\to G\to Q\to 1$ of compactly
  generated groups, where $G$ is endowed with a word-length $\ell_G$
  corresponding to a relatively compact generating set $S$, and $Q$ is
  endowed with the word-length $\ell_Q$ corresponding to $\pi(S)$, the
  group $G$ is identified with $N\times Q$ using a cross-section
  $\sigma\colon Q\to G$ satisfying $\ell_G(\sigma(q))=\ell_Q(q)$ for
  all $q\in Q$, then for all $(n,q)\in G$ we have
  \begin{equation}
    \ell_G(n)+\ell_Q(q) \leq 3 \ell_G(n,q).
  \end{equation}
\end{lemma}

\begin{proof}
  Let $(n,q)\in G$. We have
  \begin{multline}
    \ell_G(n)+\ell_Q(q) = \ell_G((n,q)\sigma(q)^{-1}) + \ell_Q(q) \leq\\\leq
    \ell_G(n,q) + 2\ell_Q(q) \leq 3\ell_G(n,q).\qedhere
  \end{multline}
\end{proof}

All these considerations sum up to the proof of our main theorem.

\begin{proof}[Proof of Theorem~\ref{thm:main-thm}]
  Suppose that $N$ has property RD with respect to the restriction of
  $\ell_G$, and $Q$ has property RD with respect to $\ell_Q$. By
  Remark~\ref{thm:good-section} there exists a cross-section
  $\sigma\colon Q\to G$ such that $\ell_G(\sigma(q))=\ell_Q(\sigma)$
  for all $q\in Q$. Then, by Lemma~\ref{thm:final-length-inequality}, the
  inequality $\ell_N(n)+\ell_Q(q) \leq 3\ell_G(n,q)$ is
  satisfied. Moreover, by Corollary~\ref{thm:unimodularity-extension},
  the group $G$ is unimodular. Hence,
  Lemma~\ref{thm:length-on-extension} applies, and $G$ has property RD.
\end{proof}

\section{Final remarks}
\label{sec:final-remarks}

In Theorem~\ref{thm:main-thm} we required $N$ to have
property RD with respect to the restriction of the word-length on
$G$. It might be tempting to ask whether this can be weakened to
having property RD with respect to a word-length on $N$. Such a
strengthening is easily seen to be false. It
suffices to consider any metabelian group with exponential growth,
e.g.\ the Baumslag-Solitar group $\langle a,b \mid
bab^{-1}=a^2\rangle$. Such a group is amenable, and by
\cite[Corollary 3.1.8]{Jolissaint1990}, no amenable group with
superpolynomial growth can satisfy property RD.

Even in the case of finitely generated groups,
Theorem~\ref{thm:main-thm} is strictly stronger than Jolissaint's
result, which assumes that $N$ has property RD with respect to its own
word-length $\ell_N$ associated to a generating set $S_N$, and the
associated functions $\theta$ and $\beta$ satisfy the inequalities
\begin{equation}
  \label{eq:poly-amp}
  \ell_N(\beta(p,q)) \leq A(1+\ell_Q(p))^\alpha(1+\ell_Q(q))^\alpha
\end{equation}
and
\begin{equation}
  \label{eq:poly-growth}
  \ell_N(\theta(q)(s)) \leq B(1+\ell_Q(q))^\beta
\end{equation}
for some $A,B,\alpha,\beta>0$. As the next proposition shows,
existence of a cross-section $\sigma$ for which $\beta$ and $\theta$
satisfy these conditions is equivalent to $N$ being \emph{polynomially
distorted} in $G$, i.e.\ to the estimate
\begin{equation}
  \label{eq:poly-dist}
  \ell_N(n)\leq C(1+\ell_G(n))^r
\end{equation}
for some $C,r>0$.  

\begin{proposition}
  \label{prop:polynomial-distortion}
  For a short exact sequence $1\to N\to G\to Q\to 1$ of finitely
  generated groups the following conditions are equivalent
  \begin{enumerate}
  \item there exists a section $\sigma\colon Q\to G$ of $\pi$ such
    that the corresponding functions $\beta$ and $\theta$ satisfy conditions~\eqref{eq:poly-amp} and \eqref{eq:poly-growth};
  \item $N$ has polynomial distortion in $G$.
  \end{enumerate}
\end{proposition}

\begin{proof}
  To show $(1)\then (2)$, assume that inequalities~\eqref{eq:poly-amp}
  and~\eqref{eq:poly-growth} hold for functions $\beta$ and $\theta$
  associated with a section $\sigma$. Observe that these conditions
  are independent of the choice of particular finite generating sets
  and corresponding word-lengths. Hence, without loss of generality,
  we may choose generating sets satisfying
  $S_G=S_N\cup\sigma(S_Q)$. We will prove~\eqref{eq:poly-dist} by
  induction on $\ell_G(n)$. The constants $C$ and $r$ will be fixed
  later.
  
  Take $n\in N$. If $n=1$, inequality~\eqref{eq:poly-dist} is
  satisfied for any $C$ and $r$, so assume that $n\ne 1$, and let
  $n=s_1s_2\cdots s_k$ be a minimal representation in $S_G$. First,
  suppose that $s_1,\dots,s_k\in \sigma(S_Q)$, and write
  $s_i=\sigma(q_i)$ with $q_i\in S_Q$. Observe that $q_1\cdots q_k=\pi(n)=1$, and therefore
  \begin{equation}
    \label{eq:polydist-proof-3}
    \begin{split}
      s_1\cdots s_k & = \left(\prod_{i=1}^{k-1} \beta(q_1\cdots q_i,
        q_{i+1})\right) \sigma(q_1\cdots q_k) \\
      & =\left(\prod_{i=1}^{k-1}
        \beta(q_1\cdots q_i, q_{i+1})\right) \sigma(1),
    \end{split}
  \end{equation}
  so 
  \begin{equation}
    \label{eq:eq-polydist-proof-3a}
    \ell_N(n)\leq \sum_{i=1}^{k-1} A(1+i)^\alpha 2^\alpha +\ell_N(\sigma(1)) \leq D_1(1+k)^{\gamma_1}.
  \end{equation}

  Now, assume that $s_1,\dots,s_m \in \sigma(S_Q)$ and $s_{m+1}\in
  S_N$ for some $m<k$. We then have
  \begin{equation}
    \label{eq:polydist-proof-1}
    n=(s_1\cdots s_m)s_{m+1}(s_1\cdots s_m)^{-1}n',
  \end{equation}
  with $n' = s_1\cdots \hat{s}_{m+1}\cdots s_k$ satisfying $\ell_G(n')\leq k-1$, and obtain
  \begin{equation}
    \begin{split}
      \label{eq:polydist-proof-2}
      \ell_N(n) & \leq \ell_N(n') + \ell_N((s_1\cdots
      s_m)s_{m+1}(s_1\cdots s_m)^{-1}) \\
      & \leq Ck^r + \ell_N((s_1\cdots s_m)s_{m+1}(s_1\cdots
      s_m)^{-1}).
    \end{split}
  \end{equation}
  To estimate the second summand, we proceed similarly as in the previous
  case. Write $s_i=\sigma(q_i)$ for $i=1,\dots,m$ and observe that,
  using the first part of~\eqref{eq:polydist-proof-3}, we get
  \begin{equation}
    \label{eq:polydist-proof-4}
    \begin{split}
      \ell_N((s_1\cdots s_m)&s_{m+1}(s_1\cdots s_m)^{-1}) \leq\\
      &\leq2\sum_{i=1}^{m-1}\ell_N(\beta(q_1\cdots q_i,q_{i+1})) +
      \ell_N(\theta(q_1\cdots q_m)(s_{m+1}))\leq\\
      &\leq 2\sum_{i=1}^{m-1} A(1+i)^\alpha 2^\alpha +
      B(1+m)^\beta\leq D_2k^{\gamma_2}. 
    \end{split}
  \end{equation}
  If we take $C=\max\{D_1,D_2\}$ and $r=1+\max\{\gamma_1,\gamma_2\}$, we finally get
  \begin{equation}
    \label{eq:polydist-proof-5}
    \ell_N(n)\leq Ck^r+D_2k^{r-1} \leq C(1+k)^r,
  \end{equation}
  which, together with~\eqref{eq:eq-polydist-proof-3a}, ends the inductive step and the proof of $(1)\then (2)$.

  In order to prove the implication $(2)\then (1)$, suppose that $N$
  is polynomially distorted in $G$, and take a section $\sigma\colon
  Q\to G$ and generating sets $S_N$, $S_G$, and $S_Q$, such that
  $S_G=S_N\cup \sigma(S_Q)$ and $\ell_G(\sigma(q))=\ell_Q(q)$ for all
  $q\in Q$. We have
  \begin{equation}
    \label{eq:polydist-proof-6}
    \begin{split}
      \ell_N(\beta(p,q)) & \leq
      C(1+\ell_G(\sigma(p)\sigma(q)\sigma(pq)^{-1}))^r \leq \\ 
      & \leq C(1 + 2\ell_Q(p)+2\ell_Q(q))^r \leq
      2^rC(1+\ell_Q(p))^r(1+\ell_Q(q))^r
    \end{split}
  \end{equation}
  and
  \begin{equation}
    \label{eq:polydist-proof-7}
    \ell_N(\theta(q)s)\leq C(1+\ell_G(\sigma(q)s\sigma(q)^{-1}))^r \leq 2^rC(1+\ell_Q(q))^r,
  \end{equation}
  which completes the proof.
\end{proof}

In particular, polynomial distortion of $N$ implies that its
word-length is dominated by the restriction of the word-length of $G$,
so the assumptions of Jolissaint can be formulated equivalently as
requiring $N$ to be polynomially distorted in $G$, and satisfy
property RD with respect to the restriction of the word-length of $G$.

In~\cite{Brady2013} and \cite{Brinkmann2000} the authors construct
hyperbolic semidirect products $N\rtimes \Z$ with $N$ free, such that
the distortion of $N$ is superpolynomial. Such extensions do not fall
into the scope of Jolissaint's theorem, while they still satisfy
assumptions of Theorem~\ref{thm:main-thm}. Indeed, since
$G=N\rtimes\Z$ is hyperbolic, it satisfies property RD by
\cite{MR972078}, and therefore $N$ has RD with respect to the
restriction of the word-length of $G$. This example can be seen as
somewhat unsatisfactory, as we already know that $G$ has RD, and we
use this to show that $N$ has RD with respect to the restricted
length.

Apart from being a subgroup of a group with property RD, the only
other criterion for having RD with respect to a length not equivalent
to a word-length that we are aware of states that an amenable group
has property RD with respect to a length $\ell$ if and only if it has
polynomial growth with respect to $\ell$. Therefore, one way to
construct a potentially nontrivial example leads through solving the
following problem.

\begin{problem}
  Construct a finitely generated group $G$ with an amenable normal
  subgroup $N$, which is superpolynomially distorted, but its relative
  growth in $G$ is polynomial, and such that the quotient $G/N$ has
  property RD.
\end{problem}
 
Of course, it would be best to construct such a group which does not
satisfy the assumptions of other known criteria for property RD, such
as hyperbolicity or admitting a proper cocompact action on a CAT(0)
cube complex.

\bibliographystyle{plain}
\bibliography{rd-ext}

\bibliographystyle{plain}
\bibliography{rd-ext.bib}
\end{document}